\documentclass[11pt]{article}
\usepackage{latexsym}
\usepackage{mathrsfs}
\usepackage{indentfirst}
\usepackage{amsmath,amsthm, amssymb}

\newtheorem{theorem}{Theorem}[section]
\newtheorem{lemma}[theorem]{Lemma}
\newtheorem{corollary}[theorem]{Corollary}
\newtheorem{remark}[theorem]{Remark}
\newtheorem{conjecture}[theorem]{Conjecture}
\begin{document}
\title{$M$-alternating Hamilton paths and $M$-alternating Hamilton cycles\thanks{Work supported by
the National Science Foundation of China and Scientific Research
Foundation of Guangdong Industry Technical College.}}
\author{Zan-Bo Zhang$^1$$^2$\thanks{Corresponding author. Email address:
eltonzhang2001@yahoo.com.cn.}, Yueping Li$^2$, Dingjun Lou$^2$,
\\ $^1$Department of Computer Engineering, Guangdong Industry Technical College \\ Guangzhou 510300, China \\
$^2$Department of Computer Science, Sun Yat-sen University \\
Guangzhou 510275, China }
\date{}
\maketitle
\begin{abstract}
We study $M$-alternating Hamilton paths and $M$-alternating
Hamilton cycles in a simple connected graph $G$ on $\nu$ vertices
with a perfect matching $M$. Let $G$ be a bipartite graph, we
prove that if for any two vertices $x$ and $y$ in different parts
of $G$, $d(x)+d(y)\geq \nu/2+2$, then $G$ has an $M$-alternating
Hamilton cycle. For general graphs, a condition for the existence
of an $M$-alternating Hamilton path starting and ending with edges
in $M$ is put forward. Then we prove that if $\kappa(G)\geq\nu/2$,
where $\kappa(G)$ denotes the connectivity of $G$, then $G$ has an
$M$-alternating Hamilton cycle or belongs to one class of
exceptional graphs. Lou and Yu \cite{LY} have proved that every
$k$-extendable graph $H$ with $k\geq\nu/4$ is bipartite or
satisfies $\kappa(H)\geq 2k$. Combining this result with those we
obtain we prove the existence of $M$-alternating Hamilton cycles
in $H$.
$\newline$ $\newline$\noindent\textbf{Key words}: degree sum,
connectivity, perfect matching, $M$-alternating path,
$M$-alternating cycle, $k$-extendable
\end{abstract}
\section{Introduction, terminologies and preliminary results}
All graphs considered in this paper are finite, undirected,
connected and simple. For the terminologies and notations not
defined in this paper, the reader is referred to \cite{BM1}.

Let $G$ be a graph with vertex set $V(G)$ and edge set $E(G)$. We
denote by $\nu$ or $|G|$ the order of $V(G)$, $\kappa$ the
connectivity of $G$, and $\delta$ the minimum degree of $G$. For
$u\in V(G)$, we denote by $d(u)$ the degree of $u$ and $N(u)$ the
set of neighbors of $u$ in $G$. For a subgraph $H$ of $G$ and a
vertex set $U\subseteq V(G-H)$, we denote by $N_{H}(U)$, or
$N_{H}(u)$ if $U$ contains only one vertex $u$, the set of
neighbors of $U$ in $H$. For any two disjoint vertex sets $X$, $Y$
in $G$ we denote by $e(X,Y)$ the number of edges of $G$ from $X$
to $Y$.

Let $C=u_{0}u_{1}\ldots u_{m-1}u_{0}$ be a cycle in $G$.
Throughout this paper, the subscripts of $u_{i}$ will be reduced
modulo $m$. We always orient $C$ such that $u_{i+1}$ is the
successor of $u_{i}$. Let $U\subseteq V(C)$ , the set of
predecessors and successors of $U$ on $C$ is denoted by $U^{-}$
and $U^{+}$ respectively, or $u^{-}$ and $u^{+}$ when $U$ contains
only one vertex $u$. For $0\leq i,j\leq m-1$, the path
$u_{i}u_{i+1}\ldots u_{j}$ is denoted by $u_{i}C^{+}u_{j}$, while
the path $u_{i}u_{i-1}\ldots u_{j}$ is denoted by
$u_{i}C^{-}u_{j}$. For a path $P=v_{0}v_{1}\ldots v_{q-1}$ and
$0\leq i,j\leq q-1$, the segment of $P$ from $v_{i}$ to $v_{j}$ is
denoted by $v_{i}Pv_{j}$.

A \emph{matching} $M$ of $G$ is a subset of $E(G)$ in which no two
elements are adjacent. If every $v\in V(G)$ is covered by an edge
in $M$ then $M$ is said to be a \emph{perfect matching} of $G$. An
\emph{$M$-alternating path} $P$ is a path of which the edges
appear alternately in $M$ and $E(G)\backslash M$. An
\emph{$M$-alternating cycle} $C$ is a cycle of which the edges
appear alternately in $M$ and $E(G)\backslash M$. We call an edge
in a matching $M$ or an $M$-alternating path starting and ending
with edges in $M$ a \emph{closed $M$-alternating path}, while an
edge in $E(G)\backslash M$ or an $M$-alternating path starting and
ending with edges in $E(G)\backslash M$ an \emph{open
$M$-alternating path}. An $M$-alternating path whose starting and
ending vertices are not covered by $M$ are called an
\emph{$M$-augmenting path}.

A graph $G$ is said to be \emph{$k$-extendable} for $0\leq k\leq
(\nu-2)/2$ if there exists a matching of size $k$ in $G$, and any
such matching is contained in a perfect matching of $G$. The
concept of $k$-extendable was introduced by Plummer in \cite{P1}.
In the same paper a relationship between extendability and
connectivity is showed.

\begin{theorem}\label{lm11}  If $G$ is a $k$-extendable
graph, then $\kappa\geq k+1$. \end{theorem}

When $k$ is large and $G$ is not bipartite, the lower bound of
connectivity can be raised.

\begin{theorem}[\textmd{Lou and Yu \cite{LY}}] \label{lm12}  If $G$
is a $k$-extendable graph with $k\geq \nu/4$, then either $G$ is
bipartite or $\kappa\geq 2k$. \end{theorem}

$M$-alternating paths and $M$-alternating cycles play important
roles in matching theory. Berge's well-known theory \cite{B1} on
maximum matchings and $M$-augmenting paths is a good
demonstration. In \cite{AHLS} and \cite{AHLZ}, $M$-alternating
paths are used to characterize $k$-extendable and
$n$-factor-critical graphs. In this paper, we study the existence
of $M$-alternating Hamilton paths and $M$-alternating Hamilton
cycles in graphs with a perfect matching. The following two lemmas
will be useful to obtain our main results.

\begin{lemma}\label{lm13} Let $G$ be a graph with a
perfect matching $M$. Let $C=u_{0}u_{1}\ldots u_{2m-1}u_{0}$ be a
longest $M$-alternating cycle in $G$, where $u_{2i-1}u_{2i}\in M$,
$0\leq i\leq m-1$. Let $v$, $w$ be the endvertices of a closed
$M$-alternating path in $G-C$. For any vertex set
$\{u_{2i},u_{2i+1}\}$, $0\leq i\leq m-1$, if $G$ is bipartite then
$e(\{u_{2i},u_{2i+1}\},\{v,w\})\leq 1$, otherwise
$e(\{u_{2i},u_{2i+1}\},\{v,w\})\leq 2$. \end{lemma}

\begin{proof}
Let $P$ be a closed $M$-alternating path connecting $v$ and $w$ in
$G-C$. If $u_{2i}v$, $u_{2i+1}w\in E(G)$, then
$u_{2i}vPwu_{2i+1}C^{+}u_{2i}$ is an $M$-alternating cycle longer
than $C$, contradicting the maximality of $C$. Thus $|\{u_{2i}v,
u_{2i+1}w\}\cap E(G)|\leq 1$. Similarly $|\{u_{2i}w,
u_{2i+1}v\}\cap E(G)|\leq 1$. So
$e(\{u_{2i},u_{2i+1}\},\{v,w\})\leq 2$. If $G$ is bipartite, then
$|\{u_{2i}v, u_{2i+1}w\}\cap E(G)|=0$ or $|\{u_{2i}w,
u_{2i+1}v\}\cap E(G)|=0$, so $e(\{u_{2i},u_{2i+1}\},\{v,w\})\leq
1$. \end{proof}

\begin{lemma}\label{lm14} Let $G$ be a graph with a
perfect matching $M$. Let $P=u_{0}u_{1}\ldots u_{2p-1}$ be a
longest closed $M$-alternating path in $G$. Let $v$, $w$ be the
endvertices of a closed $M$-alternating path in $G-P$. For any
vertex set $\{u_{2i-1},u_{2i}\}$, $1\leq i\leq p-1$, if $G$ is
bipartite then $e(\{u_{2i-1},u_{2i}\},\{v,w\})\leq 1$, otherwise
$e(\{u_{2i-1},u_{2i}\},\{v,w\})\leq 2$. \end{lemma}

\begin{proof}The proof is similar to that of Lemma \ref{lm13}. \end{proof}

\section{$M$-alternating cycles in bipartite graphs}

\begin{theorem}\label{th21}
Let $G$ be a bipartite graph and $M$ a perfect matching of $G$.
For any two vertices $x$ and $y$ in different parts of $G$,
$d(x)+d(y)\geq \nu/2+2$. Then $G$ has an $M$-alternating Hamilton
cycle.
\end{theorem}

\begin{proof}
Let $G^\prime$ be a graph, with a perfect matching $M$, which
satisfies the conditions of the theorem but does not have an
$M$-alternating Hamilton cycle. We add edges to $G^\prime$ until
the addition of any more edge results in an $M$-alternating
Hamilton cycle. Let the graph obtained finally be $G$.

Let the bipartition of $G$ be $(A,B)$. $G$ cannot be complete
bipartite, or an $M$-alternating Hamilton cycle exists. So there
are two nonadjacent vertices $w_0\in A$ and $w_{\nu-1}\in B$. By
our assumption on $G$, $G+w_0w_{\nu-1}$ has an $M$-alternating
Hamilton cycle. Hence, there is a closed $M$-alternating Hamilton
path in $G$ connecting $w_0$ and $w_{\nu-1}$. Let the path be
$P^\prime=w_0w_1\ldots w_{\nu-1}$, where $w_{2i}\in A$ and
$w_{2i-1}\in B$, $0\leq i\leq \nu/2$. Since
$d(w_0)+d(w_{\nu-1})\geq \nu/2+2$, without loss of generality, let
$d(w_0)\geq d(w_{\nu-1})$, we have $d(u_0)\geq \nu/4+1$. Hence the
neighbor $w_i$ of $w_0$ with the maximum subscript $i$ satisfies
$i\geq 2(\nu/4+1)=\nu/2+2$. Then $w_0P^\prime w_iw_0$ is an
$M$-alternating cycle with length at least $\nu/2+2$.

Let $C=u_{0}u_{1}\ldots u_{2m-1}u_{0}$ be one longest
$M$-alternating cycle in $G$, where $u_{2i}\in A$, $u_{2i+1}\in B$
and $u_{2i-1}u_{2i}\in M$, $0\leq i \leq m-1$. Then $2m< \nu$. By
above discussion, $2m\geq \nu/2+2$. Let $G_1=G-C$, we have
$|G_1|\leq \nu/2-2$. Denote the degree of a vertex $x$ in $G_1$ by
$d_1(x)$.

Let $v_0$ be a vertex in $G_1$ who sends some edges to $C$.
Without loss of generality let $v_0\in A$. Let $P=v_0v_1\ldots
v_{2p-1}$ be a maximal closed $M$-alternating path in $G_1$
starting with $v_0$. Then $v_{2p-1}$ cannot be adjacent to any
vertex in $G_1-P$. So $d_1(v_{2p-1})\leq p$.

Assume that $v_{2p-1}$ also sends some edges to $C$. Since $G$ is
bipartite, $v_{0}$ and $v_{2p-1}$ can only be adjacent to
$u_{2i+1}$ and $u_{2j}$, $0\leq i,j\leq m-1$, respectively. Let
$u_{2r+1}$ and $u_{2s}$ be the neighbors of $v_{0}$ and $v_{2p-1}$
on $C$ such that the path $P_{1}=u_{2s}C^{+}u_{2r+1}$ is the
shortest. Then any internal vertex of $P_{1}$ cannot be adjacent
to $v_{0}$ or $v_{2p-1}$. Consider the $M$-alternating cycle
$C_{1}=u_{2r+1}C^{+}u_{2s}v_{2p-1}Pv_{0}u_{2r+1}$. Since $C$ is
the longest $M$-alternating cycle in $G$, $|C_{1}|\leq |C|$, so
$|P|\leq |P_{1}|-2$.

By Lemma \ref{lm13}, for any vertex set $\{u_{2i}, u_{2i+1}\}$ on
$P_{2}$, $e(\{u_{2i},u_{2i+1}\},\{v_{0},v_{2p-1}\})\leq 1$. The
number of such sets is $$(|P_{2}|-2)/2= (|C|-|P_{1}|+2-2)/2 \leq
(|C|-(|P|+2))/2=(|C|-|P|)/2-1.$$ So
\begin{eqnarray} d(v_{0})+d(v_{2p-1})
&=&|N_{C}(v_{0})|+|N_{C}(v_{2p-1})|+d_1(v_{0})+d_1(v_{2p-1}) \nonumber \\
& \leq& ((|C|-|P|)/2-1+2)+|G_1|/2+p \nonumber \\
&=& (2m-2p)/2+1+(\nu-2m)/2+p \nonumber \\
&=&\nu/2+1 \nonumber,
\end{eqnarray}
contradicting $d(v_{0})+d(v_{2p-1})\geq \nu/2+2$. Therefore,
$v_{2p-1}$ sends no edges to $C$. Similarly, for any vertex $x\in
G_1$ who sends some edges to $C$, and any maximal close
$M$-alternating path $P_0$ in $G_1$ starting with $x$, the other
endvertex $y$ of $P_0$ sends on edge to $C$.

We also have $d(v_{2p-1})\leq p\leq |G_1|/2\leq \nu/4-1$. For any
vertex $x\in A\cap V(G_1)$, $d(x)\geq \nu/2+2-d(v_{2p-1})\geq
\nu/2+2-(\nu/4-1)=\nu/4+3$. Since $d_1(x)\leq |G_1|/2\leq
\nu/4-1$, $x$ must send some edges to $C$.

Suppose that $y\in B\cap V(G_1)$ sends some edges to $C$. Let
$P(y)$ be a maximal closed $M$-alternating path in $G_1$ starting
with $y$. Then, the other endvertex $x$ of $P(y)$ sends on edge to
$C$. However $x\in A\cap V(G_1)$, a contradiction. So for any
$y\in B\cap V(G_1)$, $y$ sends no edge to $C$. Hence $d(y)\leq
|G_1|/2$. Correspondingly, for any $u_{2i}$, $0\leq i\leq m-1$,
$u_{2i}$ sends no edge to $G_1$, so $d(u_{2i})\leq |C|/2$. But
then $d(u_{2i})+d(y)\leq |C|/2+|G_1|/2=\nu/2$, contradicting the
conditions of our theorem. So $G$, and therefore $G^\prime$, must
have an $M$-alternating Hamilton cycle.
\end{proof}
%
\begin{remark}
The lower bound of degree sum in Theorem \ref{th21} is best
possible. Let $H_0$ and $H_1$ be two disjoint complete bipartite
with bipartition $(U_0,\ V_0)$ and $(U_1,\ V_1)$ respectively,
where $|U_0|=|U_1|=|V_0|=|V_1|$. Let $u,\ v\notin V(H_0)\cup
V(H_1)$ be two different vertices. We construct graph $G$ by
joining $u$ to every vertex in $V_i$, $v$ to every vertex in
$U_i$, $i=0,\ 1$, and $u$ to $v$. For any $x$ and $y$ in different
parts of $G$, we have $d(x)+d(y) \geq \nu/2+1$. Let $M$ be a
perfect matching containing the edge $uv$, $G$ does not have an
$M$-alternating Hamilton cycle.
\end{remark}
\section{$M$-alternating paths in general graphs}
In this section we bring forward a result on the relationship
between degree sums and $M$-alternating Hamilton paths, which will
be used in the next section as well.
\begin{theorem}\label{th31} Let $G$ be a graph with a
perfect matching $M$. For any $x$, $y\in$ $V(G)$ connected by a
closed $M$-alternating path, $d(x)+d(y)\geq\nu-1$. Then $G$ has a
closed $M$-alternating Hamilton path.\end{theorem}
\begin{proof} Suppose that $G$ does not have a closed $M$-alternating
Hamilton path. Let $P=u_{0}u_{1}\ldots$$u_{2m-1}$ be a longest
closed $M$-alternating path in $G$. Then $|P|\leq\nu-2$.

By the choice of $P$, $N(u_{0})$, $N(u_{2m-1})\subseteq V(P)$. So
$$|P|\geq max(d(u_{0}),d(u_{2m-1}))+1\geq
(d(u_{0})+d(u_{2m-1}))/2+1\geq(\nu-1)/2+1=(\nu+1)/2.$$ Let
$N_{0}(u_{0})$ and $N_{1}(u_{0})$ be the set of the neighbors of
$u_{0}$ whose indices are even and odd, $N_{0}(u_{2m-1})$ and
$N_{1}(u_{2m-1})$ be the set of the neighbors of $u_{2m-1}$ whose
indices are even and odd, respectively. Let $S=M\backslash$$E(P)$.
Denoted by $V(S)$ the set of vertices associated with the edges in
$S$. Then
\begin{equation}|N_{0}(u_{0})|+|N_{1}(u_{0})|+|N_{0}(u_{2m-1})|+|N_{1}(u_{2m-1})|
=d(u_{0})+d(u_{2m-1})\geq\nu-1.\end{equation}
\textbf{Claim 1.} There does not exist an $M$-alternating cycle
$C$ in $G$ such that $V(P)\subseteq V(C)$.

Suppose that such a cycle $C$ exists. Then for an edge $xy\in
M\backslash E(C)$, each of $x$ and $y$ cannot be adjacent to any
vertex on $C$, or we can obtain a closed $M$-alternating path
longer than $P$, by going through $xy$, then all vertices on $C$.
So
$$d(x)+d(y)\leq 2(\nu-1)-2|C|\leq 2(\nu-1)-2|P|\leq 2(\nu-1)-(\nu+1)=\nu-3,$$
contradicting the condition of the theorem. Thus Claim 1 holds.
$\Box$

For any edge $u_{2i-1}u_{2i}$, 1$\leq$$i\leq$$m-1$, if
$u_{0}u_{2i}$, $u_{2i-1}u_{2m-1}\in$$E(G)$, then we obtain an
$M$-alternating cycle $u_{0}u_{2i}Pu_{2m-1}u_{2i-1}Pu_{0}$
containing all vertices on $P$, contradicting Claim 1. So
\begin{equation}|N_{0}(u_{0})|+|N_{1}(u_{2m-1})|\leq m-1.\end{equation}
By Claim 1, $u_{0}$ and $u_{2m-1}$ cannot be adjacent to each
other, so $|N_{1}(u_{0})|\leq m-1$ and $|N_{0}(u_{2m-1})|\leq
m-1$. Together with (1), we have
\begin{equation} |N_{0}(u_{0})|+|N_{1}(u_{2m-1})|\geq
(\nu-1)-(|N_{1}(u_{0})|+|N_{0}(u_{2m-1})|)\geq \nu-2m+1
.\end{equation}

By (2) and (3), $m-1\geq \nu-2m+1$, that is,
\begin{equation}m\geq (\nu+2)/3.\end{equation}

By (1) and (2),
\begin{equation}|N_{1}(u_{0})|+|N_{0}(u_{2m-1})|\geq \nu-m.\end{equation}

We classify all sets $\{u_{2i-1},u_{2i}\}$, $1\leq i\leq m-1$ as
following. If $|\{u_{0}u_{2i-1},u_{2m-1}u_{2i}\} \cap E(G)|=0$,
$1$ or $2$, then let $\{u_{2i-1},u_{2i}\} \in \mathscr{C}_{0}$,
$\mathscr{C}_{1}$ or $\mathscr{C}_{2}$. Let
$|\mathscr{C}_{1}|=r_{1}$ and $|\mathscr{C}_{2}|=r_{2}$. Then
\begin{equation} r_{1}+r_{2} \leq m-1, \end{equation} and
\begin{equation} r_{1}+2r_{2}=|N_{1}(u_{0})|+|N_{0}(u_{2m-1})| \geq\nu-m. \end{equation}
$\indent$By (6) and (7), we have $r_{2}\geq \nu-2m+1$.
$\newline$ \textbf{Claim 2}. For any $xy\in S$, $N_{P}(x)\neq
\phi$ and $N_{P}(y)\neq \phi$.

Suppose that the claim is not true and without loss of generality
let $N_{P}(y)= \phi$. For any edge $u_{2i-1}u_{2i}$, $1\leq i\leq
m-1$, if $u_{0}u_{2i}\in E(G)$, then $x$ cannot be adjacent to
$u_{2i-1}$, or $yxu_{2i-1}Pu_{0}u_{2i}Pu_{2m-1}$ is a closed
$M$-alternating path longer than $P$, contradicting the maximality
of $P$. Similarly, if $u_{2m-1}u_{2i-1}\in E(G)$, then $x$ cannot
be adjacent to $u_{2i}$. Furthermore $x$ cannot be adjacent to
$u_{0}$ and $u_{2m-1}$. Thus $|N_{P}(x)|\leq
2m-(|N_{0}(u_{0})|+|N_{1}(u_{2m-1})|)-2 \leq
2m-(\nu-2m+1)-2=4m-\nu-3$.
$\newline\indent$Since $|N(x)\cap V(S)|\leq |V(S)|-1 = \nu-2m-1$
and similarly $|N(y)\cap V(S)|\leq \nu-2m-1$. We have
$d(x)+d(y)\leq 4m-\nu-3+2(\nu - 2m-1)\leq \nu-5$, contradicting
the condition of the theorem. So Claim 2 must hold. $\Box$
$\newline\indent$We call an edge $u_{2i-1}u_{2i}$, $1\leq i\leq
m-1$, removable if $\{u_{2i-1},u_{2i}\}\in \mathscr{C}_{2}$. For
every removable edge $u_{2i-1}u_{2i}$ we get two $M$-alternating
cycles containing all vertices of $P$, that is,
$C_{0}=u_{0}Pu_{2i-1}u_{0}$ and $C_{1}=u_{2i}Pu_{2m-1}u_{2i}$. For
any edge $xy\in S$, if $N_{C_{0}}(x)\neq \phi \neq N_{C_{1}}(y)$,
or $N_{C_{1}}(x)\neq \phi \neq N_{C_{0}}(y)$, then we obtain a
closed $M$-alternating path longer than $P$, by traversing all
vertices on $C_{0}$, followed by $x$ and $y$ and those on $C_{1}$,
contradicting the maximality of $P$. But by Claim 2, $N_{P}(x)\neq
\phi \neq N_{P}(y)$. So either $N_{P}(x), N_{P}(y)\subseteq
V(C_{0})$ or $N_{P}(x), N_{P}(y)\subseteq V(C_{1})$.
$\newline\indent$Let $r=r_{2}$, $\{e_{1}, e_{2}, \ldots ,e_{r}\}$
the set of removable edges, $P_{0}, P_{1}, \ldots, P_{r}$ the
$r+1$ segments of $P$ obtained by removing all removable edges.
Then $P=P_{0}e_{1}P_{1}e_{2}\ldots e_{r}P_{r}$ and
$V(P)=\cup_{i=0}^{r}V(P_{i})$. Note here that the length of
$P_{i}$ ($0\leq i\leq r$) is at least 1.
$\newline\indent$For any edge $xy\in S$, suppose that there exist
integers $s$, $t$, $0\leq s\neq t\leq r$, such that
$N_{P_{s}}(x)\neq \phi \neq N_{P_{t}}(y)$. Without loss of
generality, suppose that $s<t$. Let $e_{t}=u_{2h-1}u_{2h}$. Then
$x$ and $y$ are adjacent to vertices on two $M$-alternating cycles
$u_{0}Pu_{2h-1}u_{0}$ and $u_{2h}Pu_{2m-1}u_{2h}$ respectively,
contradicting our conclusion above. So there must exist an integer
$l$, $1\leq l \leq r$, such that all neighbors of $x$, $y$ on $P$
be on $P_{l}$.
$\newline\indent$Let $P_{l}=u_{2g}u_{2g+1}\ldots u_{2g+2p-1}$.
Counting the vertices on $P_{l}$, we have
$$2p=|P_{l}|=|E(P_{l})|+1\leq(|E(P)|-2r)+1=2m-2r\leq
2m-2(\nu-2m+1)=6m-2\nu-2.$$ Note that by (4) the last value is
positive. By Lemma \ref{lm14},
$e(\{x,y\},\{u_{2g+2j-1},u_{2g+2j}\})\leq 2$ for $1 \leq j\leq
p-1$. So $e(\{x,y\},\{u_{2g+1},u_{2g+2},\ldots,v_{2g+2p-2}\})\leq
2(p-1)$. Then
$$|N_{P}(x)|+|N_{P}(y)|\leq
2(p-1)+4=2p+2\leq 6m-2\nu-2+2=6m-2\nu.$$
$\indent$Since $|N(x)\cap V(S)|$, $|N(y)\cap V(S)|\leq \nu-2m-1$,
we have
\begin{eqnarray}
d(x)+d(y) &=& |N_{P}(x)|+|N_{P}(y)|+|N(x)\cap V(S)|+|N(y)\cap
V(S)| \nonumber \\
&\leq& 6m-2\nu+2(\nu-2m-1) \nonumber \\
&=& 2m-2 \nonumber \\
&<& \nu-2, \nonumber
\end{eqnarray}
again contradicting the condition of our theorem.
\end{proof}
\section{$M$-alternating cycles in general graphs}
In this section, we prove that except for one class of graphs,
every graph $G$ with $\kappa\geq\nu/2$ and a perfect matching $M$
has an $M$-alternating Hamilton cycle. Firstly we construct the
exceptional graphs.
$\newline\indent$We define $\mathcal{G}_1$ as the class of graphs
constructed by taking two copies of the complete graph $K_{2n+1}$,
$n\geq 1$, with vertex sets $\{x_{1},x_{2},\ldots, x_{2n+1}\}$ and
$\{y_{1},y_{2},\ldots, y_{2n+1}\}$, and joining every $x_{i}$ to
$y_{i}$, $1\leq i\leq 2n+1$. It is easy to check that any graph
$G\in\mathcal{G}_1$ with size $4n+2$ ($n\geq 1$) is
$(2n+1)$-connected, but if we take the perfect matching
$M=\{x_{i}y_{i}:1\leq i\leq 2n+1 \}$, then there is no
$M$-alternating Hamilton cycle in $G$. We call $M$ the jointing
matching of $G$. Note that the jointing matching of $G$ is unique.
\begin{lemma} \label{lm41} Let $G$ be a graph with
$\kappa\geq\nu/2$ and $M$ a perfect matching of $G$. Then $G$ has
an $M$-alternating cycle $C$ such that $|C|\geq \nu/2+1$.
\end{lemma}
\begin{proof} Suppose that there is no $M$-alternating cycle $C$ with
$|C|\geq \nu/2+1$ in $G$. By $\kappa\geq \nu/2$ we have
$\delta\geq \nu/2$, so $d(x)+d(y)\geq \nu$ for any $x,y\in V(G)$.
By Theorem \ref{th31}, there is an $M$-alternating Hamilton path
in $G$. Let the path be $P=u_{0}u_{1}\ldots u_{2m-1}$, where
$2m=\nu$. We follow the notations $N_{0}(u_{0})$, $N_{1}(u_{0})$,
$N_{0}(u_{2m-1})$, $N_{1}(u_{2m-1})$ in Theorem \ref{th31}.
$\newline\indent$Obviously $u_{0}u_{2m-1}\notin E(G)$, or we have
an $M$-alternating Hamilton cycle, contradicting our assumption.
For any $1\leq i\leq m-1$, if $u_{0}u_{2i}, u_{2m-1}u_{2i-1}\in
E(G)$, then $u_{0}Pu_{2i-1}u_{2m-1}Pu_{2i}u_{0}$ is an
$M$-alternating Hamilton cycle, again contradicting our
assumption. So $u_{0}u_{2i}\notin E(G)$ or $u_{2m-1}u_{2i-1}\notin
E(G)$. Hence $|N_{0}(u_{0})|+|N_{1}(u_{2m-1})|\leq \nu/2-1$.
Therefore,
$$|N_{1}(u_{0})|+|N_{0}(u_{2m-1})|=d(u_{0})+d(u_{2m-1})-(|N_{0}(u_{0})|+|N_{1}(u_{2m-1})|)\geq
\nu/2+1.$$
$\indent$Without loss of generality suppose that
$|N_{1}(u_{0})|\geq |N_{0}(u_{2m-1})|$. Then $|N_{1}(u_{0})|\geq
\nu/4+1/2$. Thus there exists an integer $l$, $1\leq l \leq m$,
such that $2l-1\geq 2(\nu/4+1/2)-1=\nu/2$ and $u_{0}u_{2l-1}\in
E(G)$. Then $u_{0}Pu_{2l-1}u_{0}$ is an $M$-alternating cycle with
length at least $\nu/2+1$, again contradicting our assumption.
\end{proof}
\begin{theorem} \label{th41} Let $G$ be a graph with
$\kappa\geq\nu/2$ and $M$ a perfect matching of $G$. Then either
$G$ has an $M$-alternating Hamilton cycle or $G\in\mathcal{G}_1$
and $M$ is the jointing matching of $G$. \end{theorem}
\begin{proof} Suppose that $G$ does not have an $M$-alternating Hamilton
cycle. Let $C=u_{0}u_{1}\ldots u_{2m-1}u_{0}$ be the longest
$M$-alternating cycle in $G$, where $u_{2i-1}u_{2i} \in M$ and
$m<\nu/2$. By $\kappa\geq\nu/2$ we have $\delta\geq\nu/2$.
$\newline\indent$Let $w\in V(G-C)$, we let
$N_{0}(w)=\{u_{2i}:u_{2i}\in N_{C}(w), 0\leq i\leq m-1\}$ and
$N_{1}(w)=\{u_{2i+1}:u_{2i+1}\in N_{C}(w), 0\leq i\leq m-1\}$. Let
$W\subseteq V(G-C)$, we let $N_{0}(W)=\{u_{2i}:u_{2i}\in N_{C}(W),
0\leq i\leq m-1\}$ and $N_{1}(W)=\{u_{2i+1}:u_{2i+1}\in N_{C}(W),
0\leq i\leq m-1\}$.
$\newline\indent$Firstly we prove that $G-C$ is connected. Suppose
to the contrary that there are at least two components in $G-C$,
say $G_{1}$ and $G_{2}$ with $|G_{1}|\leq |G_{2}|$. There is at
least one edge $v_{0}v_{1}\in M \cap E(G_{1})$. By Lemma
\ref{lm13} $e(\{u_{2i},u_{2i+1}\},\{v_{0},v_{1}\})\leq 2$ for
every $0\leq i \leq m-1$. So $|N_{C}(v_{0})|+|N_{C}(v_{1})|\leq
2m$. Let $d_{1}(v)$ denote the degree of $v\in V(G_{1})$ in
$G_{1}$. Then
$$d(v_{0})+d(v_{1})=d_{1}(v_{0})+d_{1}(v_{1})+|N_{C}(v_{0})|+|N_{C}(v_{1})|\leq
2(|G_{1}|-1)+2m \leq |G_{1}|+|G_{2}|-2+2m \leq \nu-2,$$
contradicting $d(v_{0})+d(v_{1})\geq \nu$. Hence $G-C$ is
connected. Let $G_{1}=G-C$.
$\newline\indent$Consider any closed $M$-alternating paths in
$G_{1}$ with endvertices $w$ and $z$. By Lemma \ref{lm13},
$e(\{u_{2i},u_{2i+1}\},\{w,z\})\leq 2$ for every $0\leq i\leq
m-1$. Thus
$$|N_{C}(w)|+|N_{C}(z)|\leq 2m.$$
$\indent$Since
$|N_{C}(w)|+|N_{C}(z)|+d_{1}(w)+d_{1}(z)=d(w)+d(z)\geq \nu$, we
have
$$d_{1}(w)+d_{1}(z)\geq \nu - (|N_{C}(w)|+|N_{C}(z)|) \geq
\nu-2m=|G_{1}|.$$
$\indent$Let $M_{1}=M-E(C)$, then $M_{1}$ is a perfect matching of
$G_{1}$ and any closed $M$-alternating path in $G_{1}$ is a closed
$M_{1}$-alternating path. $G_{1}$ with $M_{1}$ satisfies the
condition of Theorem \ref{th31}, so there is a closed
$M_{1}$-alternating Hamilton path in $G_{1}$, or equally, a closed
$M$-alternating path in $G$ containing all vertices in $G_{1}$.
Let such a path be $P=v_{0}v_{1}\ldots v_{2q-1}$, where
$2q=\nu-2m$. We have the following cases to discuss.
$\newline$\textbf{Case 1.} There exist $r$, $s$, $0\leq r,s \leq
q-1$, such that there are no closed $M$-alternating path in
$G_{1}$ connecting $v_{2r}$ and $v_{2s+1}$.
$\newline\indent$Obviously $2s+1<2r$, or $v_{2r}Pv_{2s+1}$ is a
closed $M$-alternating path in $G_{1}$ connecting $v_{2r}$ and
$v_{2s+1}$. Thus we have $s<r$ and $|G_{1}|\geq 4$. Consider
$v_{2s}$ and $v_{2r+1}$. They are the endvertices of a closed
$M$-alternating path in $G_{1}$. By the discussion above,
\begin{equation}\label{eq41} d_{1}(v_{2s})+d_{1}(v_{2r+1})\geq
|G_{1}|=2q.\end{equation}
$\indent$For any vertex set $\{v_{2i},v_{2i+1}\}$, $0\leq i \leq
q-1$, $i\neq r,s$, if $v_{2s}v_{2i+1}, v_{2i}v_{2r+1}\in E(G)$,
then $$v_{2s+1}v_{2s}v_{2i+1}v_{2i}v_{2r+1}v_{2r}$$ is a closed
$M$-alternating path in $G_{1}$ connecting $v_{2r}$ and
$v_{2s+1}$, contradicting the assumption of Case 1. So
$|\{v_{2s}v_{2i+1}, v_{2i}v_{2r+1}\}\cap E(G)|\leq 1$. Similarly
$|\{v_{2s}v_{2i}, v_{2i+1}v_{2r+1}\}\cap E(G)|\leq 1$. So
$$e(\{v_{2s},v_{2r+1}\},\{v_{2i},v_{2i+1}\})\leq 2.$$
Furthermore, $v_{2s}$ and $v_{2r+1}$ cannot be adjacent or
$v_{2s+1}v_{2s}v_{2r+1}v_{2r}$ is a closed $M$-alternating path in
$G_{1}$ connecting $v_{2r}$ and $v_{2s+1}$. So
\begin{equation}\label{eq42} d_{1}(v_{2s})+d_{1}(v_{2r+1}) \leq 2(q-2) + 4=2q.\end{equation}
Thus equalities in (\ref{eq41}) and (\ref{eq42}) must hold.
Furthermore $|N_{C}(v_{2s})|+|N_{C}(v_{2r+1})|=2m$ and
$$e(\{u_{2j},u_{2j+1}\},\{v_{2s},v_{2r+1}\})= 2$$ for every $0\leq j
\leq m-1$.
$\newline\indent$We classify the sets $\{u_{2j},u_{2j+1}\}$,
$0\leq j \leq m-1$ into four classes, by the distribution of the 2
edges between $\{u_{2j},u_{2j+1}\}$ and $\{v_{2s},v_{2r+1}\}$.
That is,
$$\{u_{2j},u_{2j+1}\}\in \left\{
\begin{array}
{l l}
\mathscr{C}_{1}, & \mbox{if } u_{2j}v_{2s}, u_{2j+1}v_{2s}\in E(G), \\
\mathscr{C}_{2}, & \mbox{if } u_{2j}v_{2s}, u_{2j}v_{2r+1}\in E(G), \\
\mathscr{C}_{3}, & \mbox{if } u_{2j+1}v_{2s}, u_{2j+1}v_{2r+1}\in E(G), \\
\mathscr{C}_{4}, & \mbox{if } u_{2j}v_{2r+1}, u_{2j+1}v_{2r+1}\in
E(G).
\end{array}\right.$$
$\newline$Let $|\mathscr{C}_{i}| = t_{i}$, $1\leq i \leq 4$. We
have $t_{1}+t_{2}+t_{3}+t_{4}=m$,
$|N_{C}(v_{2s})|=t_{2}+t_{3}+2t_{1}$,
$|N_{C}(v_{2r+1})|=t_{2}+t_{3}+2t_{4}$,
$|N_{0}(v_{2s})|=t_{1}+t_{2}$, $|N_{1}(v_{2s})|=t_{1}+t_{3}$,
$|N_{0}(v_{2r+1})|=t_{2}+t_{4}$ and
$|N_{1}(v_{2r+1})|=t_{3}+t_{4}$.
$\newline$\textbf{Case 1.1.} $t_{2}$ or $t_{3}\neq 0$. Without
loss of generality let $t_{2}> 0$.
$\newline$\textbf{Case 1.1.1.} $t_{2}=m$.
$\newline\indent$For any $0\leq i < j \leq m-1$,
$u_{2i+1}u_{2j+1}\notin E(G)$, or
$u_{2j}v_{2s}Pv_{2r+1}u_{2i}C^{-}u_{2j+1}u_{2i+1}C^{+}u_{2j}$ is
an $M$-alternating cycle longer than $C$, a contradiction.
Therefore any $u_{2l+1}$, $0\leq l \leq m-1$, has at most
$|C|/2=m$ neighbors on $C$. Thus $|N_{G_{1}}(u_{2l+1})|\geq
d(u_{2l+1})-m\geq \nu /2-m= q$.
$\newline\indent$Since $u_{2l}$ is adjacent to $v_{2r+1}$,
$u_{2l+1}$ cannot be adjacent to any vertex $v_{2i}$, $0\leq
2i\leq 2r$, or $u_{2l+1}v_{2i}Pv_{2r+1}u_{2l}C^{-}u_{2l+1}$ is an
$M$-alternating cycle longer than $C$, a contradiction. Similarly,
$u_{2l+1}$ cannot be adjacent to any vertex $v_{2j+1}$, $2s+1\leq
2j+1\leq 2q-1$. So $N_{G_{1}}(u_{2l+1})\leq
2q-(r+1)-(q-s)=q-(r+1-s)\leq q-2$, contradicting
$N_{G_{1}}(u_{2l+1})\geq q$.
$\newline$ \textbf{Case 1.1.2.} $0<t_{2}<m$.
$\newline\indent$There exists an integer $h$, $0\leq h \leq m-1$,
such that $\{u_{2h},u_{2h+1}\}\in \mathscr{C}_{2}$, while
$\{u_{2h+2},u_{2h+3}\}\in \mathscr{C}_{i}$, $i=1$, 3 or 4. Then
$u_{2h+3}$ is adjacent to $v_{2s}$ or $v_{2r+1}$. Without loss of
generality assume that $u_{2h+3}v_{2s}\in E(G)$. Since $v_{2s}P
v_{2r+1}$ has length greater or equal to 3. The $M$-alternating
cycle $u_{2h}v_{2r+1}Pv_{2s}u_{2h+3}C^{+}u_{2h}$ is longer than
$C$, contradicting the maximality of $C$.
$\newline$ \textbf{Case 1.2.} $t_{2}=t_{3}=0$.
$\newline\indent$If $t_{1}\neq 0 \neq t_{4}$, then there exists an
integer $h$, $0\leq h \leq m-1$, such that $\{u_{2h},u_{2h+1}\}\in
\mathscr{C}_{1}$ and $\{u_{2h+2},u_{2h+3}\} \in \mathscr{C}_{4}$.
Similar to Case 1.1.2 we get an $M$-alternating cycle
$u_{2h}v_{2s}Pv_{2r+1}u_{2h+3}C^{+}u_{2h}$ which is longer than
$C$, a contradiction.
$\newline\indent$If $t_{1}$ or $t_{4}=0$, say $t_{1}=0$, then
$t_{4}=m$ and $N_{C}(v_{2s})=0$. By Lemma \ref{lm41}, $|C|\geq
\nu/2+1$, hence $d(v_{2s})\leq \nu-1-|V(C)| \leq \nu/2-2$,
contradicting $d(v_{2s})\geq \nu/2$.
$\newline$ \textbf{Case 2.} For any vertex set
$\{v_{2i},v_{2j+1}\}$, $0\leq i,j\leq q-1$, there is a closed
$M$-alternating path in $G_{1}$ connecting them.
$\newline\indent$Let $V_{0}=\{v_{2i}:v_{2i}\in V(P)\}$ and
$V_{1}=\{v_{2i+1}:v_{2i+1}\in V(P)\}$. For any vertex set
$\{u_{2l},u_{2l+1}\}$, $0\leq l \leq m-1$, suppose that there
exist two integers $0\leq i,j \leq q-1$, $u_{2l}v_{2i}$,
$u_{2l+1}v_{2j+1} \in E(G)$. By the condition of Case 2 there is a
closed $M$-alternating path $P_{1}$ in $G_{1}$ connecting $v_{2i}$
and $v_{2j+1}$, thus we obtain an $M$-alternating cycle
$u_{2l}v_{2i}P_{1}v_{2j+1}u_{2l+1}C^{+}u_{2l}$ which is longer
than $C$, a contradiction. Therefore $u_{2l}\notin N_{C}(V_{0})$
or $u_{2l+1}\notin N_{C}(V_{1})$. Similarly $u_{2l}\notin
N_{C}(V_{1})$ or $u_{2l+1}\notin N_{C}(V_{0})$. Hence
\begin{equation}\label{eq43}|N_{C}(V_{0})\cap\{u_{2l},u_{2l+1}\}|+|N_{C}(V_{1})\cap\{u_{2l},u_{2l+1}\}|
\leq 2\end{equation}
and \begin{equation}\label{eq44}|N_{C}(V_{0})|+|N_{C}(V_{1})|\leq
2m.\end{equation}
$\indent$We classify all sets $\{u_{2l},u_{2l+1}\}$ for which the
equality in (\ref{eq43}) holds into four classes. Let
$$\{u_{2l},u_{2l+1}\}\in \left\{
\begin{array} {l l}
\mathscr{C}_{1}, & \mbox{if } u_{2l}, u_{2l+1} \mbox{ send edges
to
} V_{0}, \\
\mathscr{C}_{2}, & \mbox{if } u_{2l} \mbox{ sends edges to } V_{0}
\mbox{ and } V_{1}, \\
\mathscr{C}_{3}, & \mbox{if } u_{2l+1} \mbox{ sends edges to }
V_{0}
\mbox{ and } V_{1}, \\
\mathscr{C}_{4}, & \mbox{if } u_{2l}, u_{2l+1} \mbox{ send edges
to } V_{1}.
\end{array}
\right.
$$
$\indent$If $|N_{C}(V_{0})|<m$, then $N_{C}(V_{0})\cup V_{1}$ is a
cut set of $G$ with size less than $q+m=\nu/2$, contradicting
$\kappa(G)\geq \nu/2$. So $|N_{C}(V_{0})|\geq m$. Similarly
$|N_{C}(V_{1})|\geq m$. We then have
$|N_{C}(V_{0})|+|N_{C}(V_{1})|\geq 2m$. By (\ref{eq44}) the
equality must hold and $|N_{C}(V_{0})|=|N_{C}(V_{1})|=m$.
Meanwhile, for every vertex set $\{u_{2l}, u_{2l+1}\}$, $0\leq l
\leq m-1$, equality in (\ref{eq43}) must hold, so $\{u_{2l},
u_{2l+1}\}\in \mathscr{C}_{i}$, $i=1,2,3$ or 4.
$\newline\indent$Let $|\mathscr{C}_{i}|=t_{i}$, $1\leq i \leq 4$.
Then $t_{1}+t_{2}+t_{3}+t_{4}=m$, $|N_{0}(V_{0})|=t_{1}+t_{2}$,
$|N_{1}(V_{0})|=t_{1}+t_{3}$, $|N_{0}(V_{1})|=t_{2}+t_{4}$,
$|N_{1}(V_{1})|=t_{3}+t_{4}$,
$2t_{1}+t_{2}+t_{3}=|N_{C}(V_{0})|=m=|N_{C}(V_{1})|=2t_{4}+t_{2}+t_{3}$
and $t_{1}=t_{4}$.
$\newline$ \textbf{Claim 1.}
$e(N_{0}(V_{0})^{+},N_{0}(V_{1})^{+})=0$ and
$e(N_{1}(V_{0})^{-},N_{1}(V_{1})^{-})=0$.

Suppose the claim does not hold and there exist integers $r$, $s$,
$g$, $h$, $0\leq r, s \leq m-1$, $0\leq g, h\leq q-1$, such that
$u_{2r}v_{2g}\in E(G)$, $u_{2s}v_{2h+1}\in E(G)$ and
$u_{2r+1}u_{2s+1} \in E(G)$. By the condition of Case 2 there is a
closed $M$-alternating path $P_{2}$ in $G_{1}$ connecting $v_{2g}$
and $v_{2h+1}$. Then
$u_{2r}v_{2g}P_{2}v_{2h+1}u_{2s}C^{-}u_{2r+1}u_{2s+1}C^{+}u_{2r}$
is an $M$-alternating cycle longer than $C$, contradicting the
maximality of $C$. Thus $e(N_{0}(V_{0})^{+},N_{0}(V_{1})^{+})=0$.
Similarly $e(N_{1}(V_{0})^{-},N_{1}(V_{1})^{-})=0$ and Claim 1
holds. $\Box$
$\newline$\textbf{Case 2.1.} $t_{2}$ or $t_{3}>0$. Without loss of
generality suppose $t_{2}>0$.
$\newline$\textbf{Case 2.1.1.} $t_{2}=m$.
$\newline\indent$The vertex set $\{u_{2i}, 0 \leq i \leq m-1\}$ is
a cut set of $G$ with size $m<\nu/2$, contradicting $\kappa(G)
\geq \nu /2$.
$\newline$\textbf{Case 2.1.2.} $0<t_{2}<m$.
$\newline\indent$There must exist an $r$, such that $\{u_{2r},
u_{2r+1}\}\in \mathscr{C}_{2}$, $\{u_{2r+2}, u_{2r+3}\}\in
\mathscr{C}_{i}$, $i=1$, 3 or 4. Hence $u_{2r+3}$ sends some edges
to $V_{0}$ or $V_{1}$. Without loss of generality, suppose
$u_{2r+3}$ sends some edges to $V_{1}$, say $u_{2r+3}v_{2g+1}\in
E(G)$, $0\leq g \leq q-1$. Let $0\leq h\leq q-1$ be such that
$u_{2r}v_{2h}\in E(G)$. By the condition of Case 2, there is a
closed $M$-alternating path $P_{3}$ in $G_1$ connecting $v_{2h}$
and $v_{2g+1}$.
$\newline\indent$Now let's estimate the sum of the degrees of
$u_{2r+1}$ and $u_{2r+2}$.
$\newline\indent$Since $\{u_{2r},u_{2r+1}\}\in \mathscr{C}_{2}$,
$u_{2r+1}$ sends no edge to $G_{1}$, the number of vertices in
which is $2q$. Since $u_{2r+3}$ sends edges to $V_{1}$,
$\{u_{2r+2},u_{2r+3}\}\in \mathscr{C}_{3}$ or $\mathscr{C}_{4}$,
so $u_{2r+2}$ sends no edge to $V_{0}$, the number of vertices in
which is $q$.

Note that $u_{2r+1} \in N_{0}(V_{0})^{+} \cap N_{0}(V_{1})^{+}$,
by Claim 1, $u_{2r+1}$ cannot be adjacent to any other vertex in
$N_{0}(V_{0})^{+} \cup N_{0}(V_{1})^{+}$, the number of which is
equal to $|N_{0}(V_{0})\cup N_{0}(V_{1})|-1$, that is,
$t_{1}+t_{2}+t_{4}-1$.

If $u_{2r+3}$ sends no edge to $V_{0}$, then $u_{2r+2}\in
N_{1}(V_{1})^{-}$ and $u_{2r+2}\notin N_{1}(V_{0})^{-}$. By Claim
1, $u_{2r+2}$ cannot be adjacent to any vertex in
$N_{1}(V_{0})^{-}$, the number of which is $t_{1}+t_{3}$. If
$u_{2r+3}$ sends some edges to $V_{0}$, then $u_{2r+2}\in
N_{1}(V_{0})^{-} \cap N_{1}(V_{1})^{-}$. Again by Claim 1,
$u_{2r+2}$ cannot be adjacent to any other vertices in $N_1(V_0)^-
\cup N_1(V_1)^-$, the number of which is equal to $t_1+t_3+t_4-1$.

Suppose there exists an integer $l$, $0\leq l \leq m-1$, $l \neq
r, r+1$, such that $u_{2l}u_{2r+1}, u_{2l+1}u_{2r+2}\in E(G)$.
Then
$$u_{2r}v_{2h}P_{3}v_{2g+1}u_{2r+3}C^{+}u_{2l}u_{2r+1}u_{2r+2}u_{2l+1}C^{+}u_{2r}$$
is an $M$-alternating cycle longer than $C$, a contradiction. Thus
for any $0\leq i\leq m-1$, $i\neq r, r+1$, $u_{2i}u_{2r+1} \notin
E(G)$ or $u_{2i+1}u_{2r+2} \notin E(G)$.

Now we can calculate an upper bound for the sum of the degrees of
$u_{2r+1}$ and $u_{2r+2}$. If $u_{2r+3}$ sends no edge to $V_{0}$,
then
\begin{eqnarray}
d(u_{2r+1})+d(u_{2r+2}) &\leq &
2(\nu-1)-2q-q-(t_{1}+t_{2}+t_{4}-1)-(t_{1}+t_{3})-(m-2) \nonumber \\
&=& 2\nu-3q-m-(t_{1}+t_{2}+t_{3}+t_{4}-1+t_1) \nonumber \\
&=& \nu+(2q+2m)-3q-m-(m-1+t_{1}) \nonumber \\
&=& \nu-(q+t_{1}-1).\nonumber
\end{eqnarray}
If $u_{2r+3}$ sends some edges to $V_{0}$, then
\begin{eqnarray}
d(u_{2r+1})+d(u_{2r+2}) &\leq &
2(\nu-1)-2q-q-(t_{1}+t_{2}+t_{4}-1)-(t_1+t_3+t_4-1)-(m-2) \nonumber \\
&=& 2\nu-3q-m-(t_{1}+t_{2}+t_{3}+t_{4}-2+t_1+t_4) \nonumber \\
&=& \nu+(2q+2m)-3q-m-(m-2+t_{1}+t_4) \nonumber \\
&=& \nu-(q+t_{1}+t_4-2).\nonumber
\end{eqnarray}

Since $d(u_{2r+1})+d(u_{2r+2})\geq \nu$ we have $(q+t_{1}-1)\leq
0$ or $(q+t_1+t_4-2)\leq 0$. But since $q\geq 1$ and $t_1=t_4\geq
0$, in both cases we have $t_{4}=t_{1}=0$. Therefore, for any
$0\leq i\leq m-1$, $\{u_{2i},u_{2i+1}\}\in \mathscr{C}_{2}\cup
\mathscr{C}_{3}$, hence $|(N_{C}(V(G_1))\cap
\{u_{2i},u_{2i+1}\}|=1$. But then $|(N_{C}(V(G_1))|= m \leq
\nu/2-1<\nu/2$ and $N_{C}(V(G_1))$ is a cut set of $G$,
contradicting $\kappa(G)\geq \nu/2$.
$\newline$ \textbf{Case 2.2.} $t_{2}=t_{3}=0$. Then
$t_{4}=t_{1}=m/2$. So $m$ must be even.
$\newline$ \textbf{Claim 2.} For a segment
$u_{2l}u_{2l+1}u_{2l+2}u_{2l+3}$ of $C$, if
$\{u_{2l},u_{2l+1}\}\in \mathscr{C}_{1}$ and
$\{u_{2l+2},u_{2l+3}\}\in \mathscr{C}_{4}$
($\{u_{2l},u_{2l+1}\}\in \mathscr{C}_{4}$ and
$\{u_{2l+2},u_{2l+3}\}\in \mathscr{C}_{1}$), then the following
statements hold.
$\newline$ (a) $|N_{G_{1}}(u_{2l})|=1$ and
$|N_{G_{1}}(u_{2l+3})|=1$. The neighbors of $u_{2l}$ and
$u_{2l+3}$ in $G_{1}$ are the endvertices of an edge in $M$.
$\newline$ (b) $u_{2l+1}$ is adjacent to all vertices in $V_{0}$
($V_{1}$) and $u_{2l+2}$ is adjacent to all vertices in $V_{1}$
($V_{0}$).
$\newline$ (c) $u_{2l+1}$ is adjacent to all other vertices in
$N_{0}(V_{0})^{+}$ ($N_{0}(V_{1})^{+}$) and $u_{2l+2}$ is adjacent
to all other vertices in $N_{1}(V_{1})^{-}$ ($N_{1}(V_{0})^{-}$).

We only prove the situation that $\{u_{2l},u_{2l+1}\}\in
\mathscr{C}_{1}$ and $\{u_{2l+2},u_{2l+3}\}\in \mathscr{C}_{4}$,
for the other situation the results follow similarly. Let
$v_{2g}\in N_{G_{1}}(u_{2l})$ and $v_{2h+1}\in
N_{G_{1}}(u_{2l+3})$, $0\leq g,h\leq q-1$. By the condition of
Case 2 there is a closed $M$-alternating path $P_{4}$ in $G_{1}$
connecting $v_{2g}$ and $v_{2h+1}$. If $|P_{4}|>1$, then the
$M$-alternating cycle
$u_{2l}v_{2g}P_{4}v_{2h+1}u_{2l+3}C^{+}u_{2l}$ is longer than $C$,
a contradiction. So $P_{4}$ consists of exactly one edge in $M$
and $g=h$. Since $v_{2g}$ and $v_{2h+1}$ is randomly chosen we
have $|N_{G_{1}}(u_{2l})|=1$ and $|N_{G_{1}}(u_{2l+3})|=1$, thus
(a) is proved.
$\newline\indent$Similar to Case 2.1.2 we count the sum of the
degrees of $u_{2l+1}$ and $u_{2l+2}$. Since
$\{u_{2l},u_{2l+1}\}\in \mathscr{C}_{1}$, $u_{2l+1}$ cannot send
any edge to $V_{1}$, so $|N_{G_{1}}(u_{2l+1})|\leq q$. Similarly
$|N_{G_{1}}(u_{2l+2})|\leq q$. By Claim 1, $u_{2l+1}$ cannot be
adjacent to any vertex in $N_{0}(V_{1})^{+}$, the number of which
is $t_{2}+t_{4}=m/2$, and $u_{2l+2}$ cannot be adjacent to any
vertex in $N_{1}(V_{0})^{-}$, the number of which is
$t_{1}+t_{3}=m/2$. For any $\{u_{2i},u_{2i+1}\}$ where $0\leq
i\leq m-1$, $i \neq l$, $l+1$, if $u_{2l+1}u_{2i}\in E(G)$ and
$u_{2l+2}u_{2i+1}\in E(G)$, then the $M$-alternating cycle
$u_{2l}v_{2g}v_{2g+1}u_{2l+3}C^{+}u_{2i}u_{2l+1}u_{2l+2}u_{2i+1}C^{+}u_{2l}$
is longer than $C$, a contradiction. Thus for any
$\{u_{2i},u_{2i+1}\}$, $0\leq i\leq m-1$, $i \neq l$, $l+1$,
$u_{2l+1}u_{2i}\notin E(G)$ or $u_{2l+2}u_{2i+1}\notin E(G)$.
Therefore
$$d(u_{2l+1})+d(u_{2l+2})\leq
2q+2(2m-1)-(m/2+m/2)-(m-2)=2q+2m=\nu. $$
$\indent$But $d(u_{2l+1})+d(u_{2l+2})\geq \nu/2+\nu/2=\nu$, thus
all equalities must hold. Hence $|N_{G_{1}}(u_{2l+1})|=q$ and
$|N_{G_{1}}(u_{2l+2})|=q$ and (b) holds. Meanwhile, except those
we excluded above, $u_{2l+1}$ must be adjacent to all other
vertices. Therefore $u_{2l+1}$ must be adjacent to all other
vertices in $N_{0}(V_{0})^{+}$. Similarly $u_{2l+2}$ must be
adjacent to all other vertices in $N_{1}(V_{1})^{+}$ and (c)
holds. The proof of Claim 2 is complete. $\Box$
$\newline$\textbf{Case 2.2.1.} There exists an integer $r$, $0\leq
r \leq m-1$, such that
$\{u_{2r},u_{2r+1}\},\{u_{2r+2},u_{2r+3}\}\in \mathscr{C}_{1}$.
$\indent$We can choose $r$ so that $\{u_{2r},u_{2r+1}\},
\{u_{2r+2},u_{2r+3}\}\in \mathscr{C}_{1}$ and
$\{u_{2r+4},u_{2r+5}\}\in \mathscr{C}_{4}$. By Claim 2 (c) and
(a), $u_{2r+1}u_{2r+3}\in E(G)$ and $|N_{G_{1}}(u_{2r+2})|=1$. Let
$v_{2g}$, $v_{2h_{1}+1}$ and $v_{2h_{2}+1}$ be the neighbors of
$u_{2r}$, $u_{2r+4}$ and $u_{2r+5}$ in $G_{1}$. By the condition
of Case 2, there is a closed $M$-alternating path $P_{5}$ in
$G_{1}$ connecting $v_{2g}$ and $v_{2h_{1}+1}$, and a closed
$M$-alternating path $P_{6}$ in $G_{1}$ connecting $v_{2g}$ and
$v_{2h_{2}+1}$.
$\newline\indent$If $u_{2r+2}u_{2r+5}\in E(G)$, then the
$M$-alternating cycle
$$u_{2r+2}u_{2r+1}u_{2r+3}u_{2r+4}v_{2h_{1}+1}P_{5}v_{2g}u_{2r}C^{-}u_{2r+5}u_{2r+2}$$
is longer than $C$, a contradiction. So $u_{2r+2}u_{2r+5}\notin
E(G)$. By Claim 1, we have $u_{2r+2}u_{2r+4}\notin E(G)$.
$\newline\indent$If there exists an integer $l$, $0\leq l\leq
m-1$, $l\neq r+2$, such that $\{u_{2l},u_{2l+1}\}\in
\mathscr{C}_{4}$ and $u_{2r+2}u_{2l+1}\in E(G)$. By Claim 2,
$u_{2r+4}u_{2l}\in E(G)$. Then the $M$-alternating cycle
$$u_{2r+2}u_{2r+1}u_{2r+3}u_{2r+4}u_{2l}C^{-}u_{2r+5}v_{2h_{2}+1}P_{6}v_{2g}u_{2r}C^{-}u_{2l+1}u_{2r+2}$$
is longer than $C$, a contradiction. Thus for all
$\{u_{2l},u_{2l+1}\}\in \mathscr{C}_{4}$, $u_{2r+2}u_{2l+1}\notin
E(G)$. But since $u_{2l}\in N_{1}(V_{1})^{-}$ and $u_{2r+2}\in
N_{1}(V_{0})^{-}$, by Claim 1, we also have $u_{2r+2}u_{2l}\notin
E(G)$. Therefore $u_{2r+2}$ has at most $2m-1-m=m-1$ neighbors on
$C$. Thus $d(u_{2r+2})\leq m-1+1=m < \nu/2$, contradicting
$d(u_{2r+2})\geq\kappa \geq \nu/2$.
$\newline$ \textbf{Case 2.2.2.} There does not exist any integer
$i$, $0\leq i \leq m-1$, such that
$$\{u_{2i},u_{2i+1}\},\{u_{2i+2},u_{2i+3}\}\in \mathscr{C}_{1}.$$
$\indent$Since $t_{1}=t_{4}=m/2$, there can neither be any $j$,
$0\leq j \leq m-1$, such that
$$\{u_{2j},u_{2j+1}\},\{u_{2j+2},u_{2j+3}\}\in \mathscr{C}_{4}.$$
Thus the sets $\{u_{2i},u_{2i+1}\}$, $0\leq i \leq m-1$ belong to
$\mathscr{C}_{1}$ and $\mathscr{C}_{4}$ alternatively. Without
loss of generality suppose $\{u_{0},u_{1}\}\in\mathscr{C}_{1}$,
then $\{u_{4i},u_{4i+1}\}\in\mathscr{C}_{1}$ and
$\{u_{4i+2},u_{4i+3}\}\in\mathscr{C}_{4}$, for $0\leq i\leq
m/2-1$. Consider the segment $u_{4i}u_{4i+1}u_{4i+2}u_{4i+3}$. By
Claim 2 (b), $u_{4i+2}$ is adjacent to all vertices in $V_{1}$.
Consider the segment $u_{4i+2}u_{4i+3}u_{4i+4}u_{4i+5}$. By Claim
2 (a), $u_{4i+2}$ can have only one neighbor in $G_{1}$. Thus we
have $|G_{1}|=2$. $G_{1}$ consists of the edge $v_{0}v_{1}\in M$
only. $N_{C}(v_{0})=\{u_{4i}, u_{4i+1}: 0\leq i \leq m/2-1\}$ and
$N_{C}(v_{1})=\{u_{4i+2}, u_{4i+3}:0\leq i \leq m/2-1\}$.
$\newline\indent$For any segment $u_{4i}u_{4i+1}u_{4i+2}u_{4i+3}$
of $C$, we obtain another longest $M$-alternating cycle
$$C^{\prime}=u_{4i}v_{0}v_{1}u_{4i+3}C^{+}u_{4i}.$$ Let
$G_{1}^{\prime}=G-C^{\prime}$, which consists of the edge
$u_{4i+1}u_{4i+2}$ only. Note that when we get here, we have
dismissed all other cases. Therefore, $C^{\prime}$ and
$G_{1}^{\prime}$ must have structures similar to $C$ and $G_{1}$,
as we have stated in this case. Hence the vertices in the sets
$\{u_{4i},v_{0}\}, \{v_{1}, u_{4i+3}\}$ and $\{u_{2j},u_{2j+1}\}$,
$0\leq j\leq m-1$, $j\neq 2i, 2i+1$, are adjacent to $u_{4i+1}$
and $u_{4i+2}$ alternatively, according to their orders on
$C^{\prime}$. Thus we have $N(u_{4i+1})=\{u_{4i+2},u_{4i},
v_{0}\}\cup \{u_{4j},u_{4j+1}:0\leq j\leq m/2-1,j\neq i\}$ and
$N(u_{4i+2})=\{u_{4i+1},u_{4i+3}, v_{1}\}\cup
\{u_{4j+2},u_{4j+3}:0\leq j\leq m/2-1,j\neq i\}$. Analogous
discussion on any segment $u_{4i-2}u_{4i-1}u_{4i}u_{4i+1}$ and
$u_{4i+2}u_{4i+3}u_{4i+4}u_{4i+5}$ leads to the conclusion that
$N(u_{4i})=\{u_{4i-1},u_{4i+1}, v_{0}\}\cup
\{u_{4j},u_{4j+1}:0\leq j\leq m/2-1,j\neq i\}$ and
$N(u_{4i+3})=\{u_{4i+4},u_{4i+2}, v_{1}\}\cup
\{u_{4j+2},u_{4j+3}:0\leq j\leq m/2-1,j\neq i\}$.
$\newline\indent$ By the arbitrariness of $i$, we conclude that
all vertices $u_{4i}$ and $u_{4i+1}$, $ 0\leq i \leq m/2-1$, are
adjacent to each other. They, together with $v_{0}$, form a
complete graph $K_{m+1}$. Similarly, vertices $u_{4i+2}$ and
$u_{4i+3}$, $ 0\leq i \leq m/2-1$, with $v_{1}$, form a complete
graph $K_{m+1}$. These two complete graphs, together with the
edges in $M$, constitute $G$. Since $m$ is even, let $m=2n$ then
$|G|=4n+2$. Therefore $G\in \mathcal{G}_1$ and $M$ is exactly the
jointing matching.
\end{proof}
\begin{corollary} Let $G$ be a $k$-extendable graph with
$k\geq\nu/4$, and $M$ a perfect matching of $G$. Then $G$ has an
$M$-alternating Hamilton cycle. \end{corollary}
\begin{proof} By Theorem \ref{lm12}, either $G$ is bipartite or $\kappa\geq
2k$. If $G$ is bipartite, then by Theorem \ref{lm11}, $\delta \geq
\kappa\geq k+1\geq \nu/4+1$. Hence, for any two vertices $x$ and
$y$ in different parts of $G$, $d(x)+d(y)\geq \nu/2+2$. By Theorem
\ref{th21}, $G$ has an $M$-alternating Hamilton cycle. If
$\kappa\geq 2k\geq \nu/2$, then by Theorem \ref{th41}, $G$ has an
$M$-alternating Hamilton cycle or $G\in \mathcal{G}_1$. If $G\in
\mathcal{G}_1$, then $|G|=4n+2$, $n\geq 1$, so $k\geq n+1$. Thus
$\kappa\geq 2k\geq 2n+2$. But $G$ is regular with degree $2n+1$, a
contradiction. So $G$ has an $M$-alternating Hamilton cycle.
\end{proof}

\section{Final Remark}
Theorem \ref{th41} is a special case of the following conjecture.
\begin{conjecture}(Lov\'{a}sz-Woodall) Let $L$ be a set of
$k$ independent edges in a $k$-connected graph $G$, if $k$ is even
or $G-L$ is connected, then $G$ has a cycle containing all the
edges of $L$.
\end{conjecture}
Professor Kawarabayashi has published \cite{KK1}, which is the
first step towards a solution for the conjecture. He is still
working for a whole proof of the conjecture when we finish the
current paper.

\section*{Acknowledgments}
We thank the referees for their careful reading and valuable
suggestions that help improving the paper.

\end{document}